\documentclass[12pt]{article}
\usepackage{amssymb}
\usepackage{amsmath}
\usepackage{graphicx}
\usepackage{fullpage}
\usepackage{amsthm}
\usepackage{shuffle}
\usepackage{natbib}
\usepackage{float}
\newcommand{\ses}{\, = \,}

\makeatletter
\newtheorem*{rep@theorem}{\rep@title}
\newcommand{\newreptheorem}[2]{
\newenvironment{rep#1}[1]{
\def\rep@title{#2 \ref{##1}}
\begin{rep@theorem}}
{\end{rep@theorem}}}
\makeatother

\newtheorem{thm}{Theorem}
\newtheorem{lemma}{Lemma}
\newtheorem{conj}{Conjecture}

\newtheorem{prop}{Proposition}
\newreptheorem{thm}{Theorem}

\numberwithin{equation}{section}
\numberwithin{thm}{section}
\numberwithin{lemma}{section}
\numberwithin{conj}{section}
\numberwithin{rmk}{section}
\numberwithin{prop}{section}

\title{A refinement of the Shuffle Conjecture\\ with cars of two sizes and $t=1/q$}
\author{Angela Hicks and Emily Leven}

\begin{document}

\maketitle

\begin{abstract}
The original Shuffle Conjecture of Haglund et al.\ has a symmetric function side and a combinatorial side. The symmetric function side may be simply expressed as
$\big\langle \nabla e_n \, , \, h_{\mu} \big\rangle$
where $\nabla$ is the \hbox{Macdonald} polynomial eigen-operator of Bergeron and Garsia and $h_\mu$ is the homogeneous basis indexed by $\mu=(\mu_1,\mu_2,\ldots ,\mu_k) \vdash n$. The combinatorial side q,t-enumerates a family of Parking Functions whose reading word is a shuffle of $k$ successive segments 
of $123\cdots n$ of respective lengths $\mu_1,\mu_2,\ldots ,\mu_k$. It can be shown that for $t=1/q$ the symmetric function side reduces to a product of $q$-binomial coefficients and powers of $q$. This reduction suggests a surprising combinatorial refinement of the general Shuffle Conjecture. Here we prove this refinement for $k=2$ and $t=1/q$. The resulting formula gives a $q$-analogue of the well studied Narayana numbers.
\end{abstract}

\section{Introduction}

A Dyck path in the $n \times n$ lattice square starts at the southwest corner of the square and proceeds to the northeast corner with $n$ north edges and $n$ east edges, always remaining weakly above the diagonal joining these the same two corners. Here and after we will refer to the cells crossed by this diagonal as in the \emph{main diagonal} of the square.

Here we visualize a Parking Function as a Dyck path in the $n \times n$ lattice square whose north steps are labeled in a column increasing way by the integers $\{1,2,\dots,n\}$. For convenience we place the label of a north edge in the cell immediately to the east of that edge. This visual representation has its origins in \cite{GH}, where it is used as a geometric way of depicting preference functions that park the cars on a one way street (see \cite{KW}). We will also borrow the term \emph{cars} for the labels of the north edges.

For computational convenience, Parking Functions may also be represented as two line arrays:
\begin{displaymath}
PF \ses \Big[ {v_1 \, v_2 \, \cdots \, v_n \atop u_1 \, u_2 \, \cdots \, u_n} \Big]
\end{displaymath}
with $u_1,u_2,\dots,u_n$ integers satisfying
\begin{displaymath} 
u_1 = 0 \hbox{ \hskip 12pt and \hskip 12pt } 0 \leq u_i \leq u_{i-1} + 1
\end{displaymath}
and $V = (v_1,v_2,\dots,v_n)$ a permutation in $S_n$ satisfying $u_i = u_{i-1}+1 \Rightarrow v_i > v_{i-1}$. The component $u_i$ represents the number of full cells in the $i$th row of the diagram which are east of the path and strictly west of the main diagonal. The component $v_i$ gives the car that resides in the $i^{th}$ row. An example of these two corresponding representations is given below.

\begin{figure}[H]
\begin{displaymath}
\left[\begin{matrix}
4 & 6 & 8 & 1 & 3 & 2 & 7 & 5 \cr
0 & 1 & 2 & 2 & 3 & 0 & 1 & 1 \cr
\end{matrix}\right]
\hskip 20pt \Longleftrightarrow \hskip 20pt
\vcenter{\hbox{\includegraphics[width=1.5in]{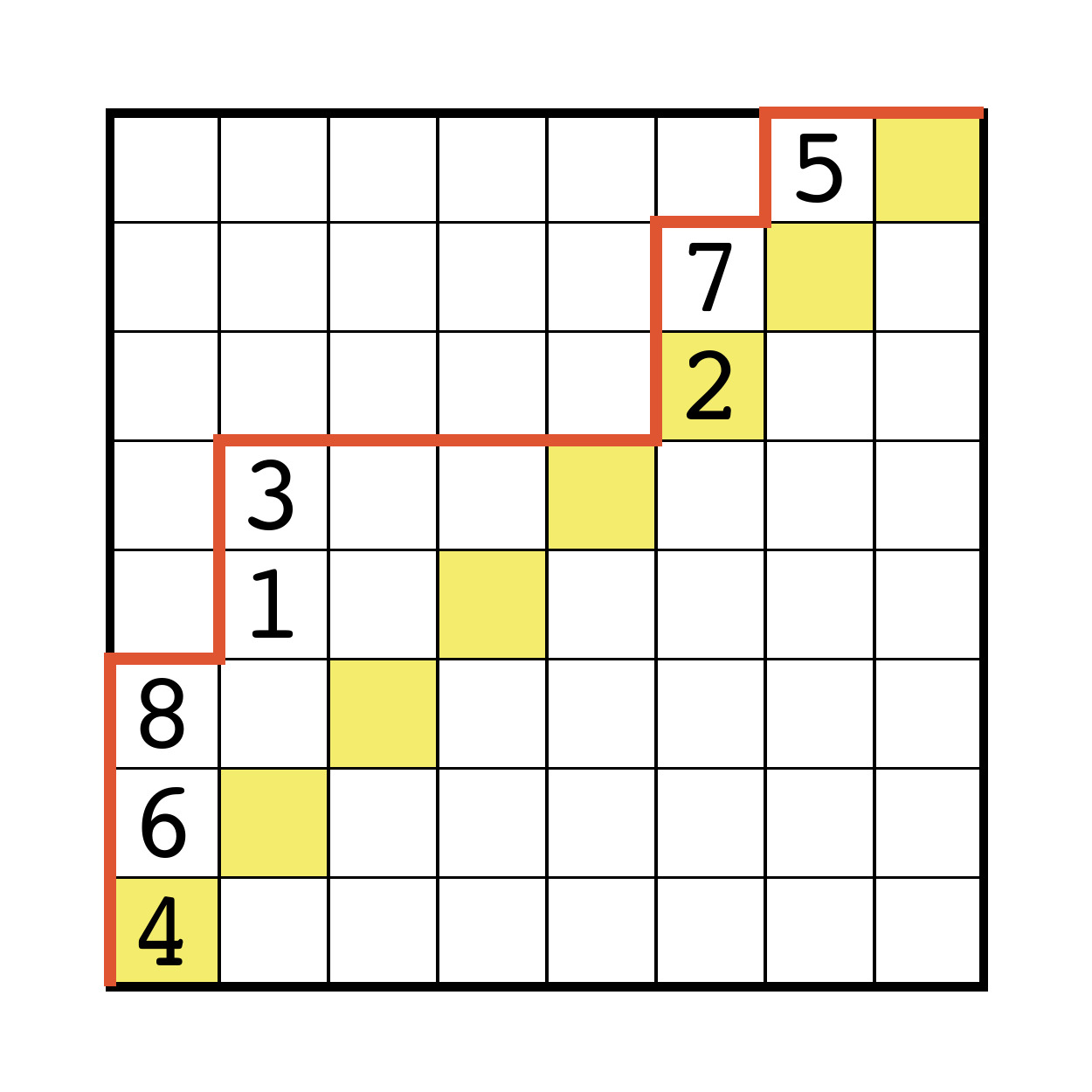}}}
\end{displaymath}
\end{figure}

We will denote by $\sigma(PF)$ the permutation obtained by successive right to left readings of the components of the vectors $V = (v_1,v_2,\dots,v_n)$ according to decreasing values of $u_1,u_2,\dots,u_n$. We will call $\sigma(PF)$ the \emph{diagonal word} of $PF$. This given, each Parking Function is assigned two statistics
\begin{equation}
\operatorname{area}(PF) \ses \sum_{i=1}^n u_i,
\end{equation}
and
\begin{equation}
\operatorname{dinv}(PF) = \sum_{1 \leq i < j \leq n} \chi(u_i = u_j \, \& \, v_i < v_j) + \sum_{1 \leq i < j \leq n} \chi(u_i = u_j + 1 \, \& \, v_i > v_j).
\end{equation}
It is easily seen that $\operatorname{area}(PF)$ gives the total number of cells between the supporting Dyck path and the main diagonal. The notation $\operatorname{dinv}(PF)$ is an abbreviation of the term ``diagonal inversions.'' Note that two cars in the same diagonal with the car on the left smaller than the car on the right will contribute a unit to $\operatorname{dinv}(PF)$ called a \emph{primary dinv}. Likewise, a car on the left that is bigger than a car on the right with the latter in the adjacent lower diagonal contributes a unit to $\operatorname{dinv}(PF)$ called a \emph{secondary dinv}. In the example above, $\sigma(PF) = \, 3 \, 1 \, 8 \, 5 \, 7 \, 6 \, 2 \, 4$, $\operatorname{area}(PF) = 10$ and $\operatorname{dinv}(PF) = 4$. 

For two integers $a+b=n$, let us refer to $1,2,\ldots ,a$ as the \emph{small} cars and $a+1,a+2,\ldots ,n$ as the \emph{big} cars and let ${\cal PF}_{a,b}^{(r,s)}$ denote the collection of Parking Functions whose reading word is a shuffle of the two words $1\, 2\, \cdots \, a$ and $a+1\, a+1 \, \dots \, n$ and which have $r$ small cars and $s$ big cars in the main diagonal. Similarly, let ${\cal PF}_{a,b}^{(s)} = \cup_{r=0}^{a} {\cal PF}_{a,b}^{(r,s)}$ denote the collection of Parking Functions whose reading word is a shuffle of the two words $1\, 2\, \cdots \, a$ and $a+1\, a+1 \, \dots \, n$ and which have $s$ big cars in the main diagonal, and let ${\cal PF}_{a,b} = \cup_{r=0}^{a} \cup_{s=0}^{b} {\cal PF}_{a,b}^{(r,s)}$.

\newpage
Our main result here can be stated as follows
\begin{thm} \label{qara}
For all $0 \leq s < b$ and $0 \leq r \leq a$ we have 
\begin{align*}
\sum_{PF \in {\cal PF}_{a,b}^{(r,s)} } &q^{\operatorname{coarea}(PF)+\operatorname{dinv}(PF)}\ses \cr
& \hskip -20pt \ses q^{{a+b \choose 2}-(a-r+1)(b-s)} 
\Big[ {a+b-s-1 \atop a} \Big]_q 
\Big[ {a-r+b-s \atop a-r} \Big]_q 
\Big[ {r+s \atop s} \Big]_q 
{[r]_q\over [a-r+b-s]_q} \cr
\end{align*}
and for all $0\leq b$, $0 \leq r \leq a$ we have
\begin{displaymath}
\sum_{PF \in {\cal PF}_{a,b}^{(r,b)} } q^{\operatorname{coarea}(PF)+\operatorname{dinv}(PF)} 
\ses 
 \chi(a=r) \, q^{a+b \choose 2} \Big[ {a+b \atop a} \Big]_q
\end{displaymath}
where $\operatorname{coarea}(PF) = {n \choose 2} - \operatorname{area}(PF)$.
\end{thm}

To see why this constitutes a refinement of the $t=1/q$ and $k=2$ case of the Shuffle conjecture we need to review some background. A more thorough introduction to the relevant tools in symmetric function theory can be found in section \ref{symsec}. To begin, the Shuffle conjecture can be stated as follows.

\begin{conj} \label{shuffleconj}
For $\mu=(\mu_1,\mu_2,\ldots ,\mu_k) \vdash n$ we have
\begin{equation}
\big\langle \nabla e_n \, , \, h_\mu \big\rangle \ses 
\sum_{PF\in {\cal PF}_n} t^{\operatorname{area}(PF)} q^{\operatorname{dinv}(PF)}
\chi \big(\sigma(PF) \in E_1\shuffle E_2 \shuffle \cdots \shuffle E_k \big)
\end{equation}
where ``${\cal PF}_n$'' denotes the collection of Parking Functions in the $n\times n$ square, $E_1,E_2,\ldots ,E_k$ are successive segments of the word $123\cdots n$ of respective lengths $\mu_1,\mu_2,\ldots ,\mu_k$ and ``$ \, \shuffle$'' denotes shuffling.
\end{conj}

Now it was already noticed by \cite{GH} that
\begin{equation} \label{GH}
q^{n\choose 2}\nabla e_n\Big|_{t=1/q}\ses {1\over [n+1]_q}e_n\big[X[n+1]_q\big].
\end{equation}
Using this identity, it is not difficult to derive that Conjecture $\ref{shuffleconj}$ implies the following, also open, conjecture.

\begin{conj} \label{shuffleconj2}
\begin{equation}
\sum_{PF\in {\cal PF}_n}q^{\operatorname{coarea}(PF)+\operatorname{dinv}(PF)}
\chi \big(\sigma(PF) \in E_1\shuffle E_2 \shuffle \cdots \shuffle E_k \big)
\ses
\frac{1}{[n+1]_q}\prod_{i=1}^k q^{\mu_i\choose 2} \Big[{n+1\atop \mu_i} \Big]_q.
\end{equation}
\end{conj}

We will show here that Theorem $\ref{qara}$ has the following corollary.

\begin{thm} \label{wolf}
For $a \geq 0$ and $b \geq 0$,
\begin{displaymath}
\sum_{PF\in {\cal PF}_{a,b}} q^{\operatorname{coarea}(PF)+\operatorname{dinv}(PF)} \ses
\frac{q^{{a\choose 2}+{b\choose 2}} }{ [a+b+1]_q}
\Big[ {a+b+1\atop a} \Big]_q
\Big[ {a+b+1\atop b} \Big]_q.
\end{displaymath}
\end{thm}

Our proofs of Theorems $\ref{qara}$ and $\ref{wolf}$ are based on two key recursions which may be stated as follows. For parameters $a,b,r,s$, let us set
\begin{equation}
\mathit{Parkq}_{a,b}^{(r,s)}(q) \ses 
\sum_{PF\in {\cal PF}_{a,b}^{(r,s)} } q^{\operatorname{coarea}(PF)+\operatorname{dinv}(PF)}
\end{equation}
and
\begin{equation}
\mathit{Parkq}_{a,b}^{(s)}(q) \ses 
\sum_{PF \in {\cal PF}_{a,b}^{(s)}} q^{\operatorname{coarea}(PF)+\operatorname{dinv}(PF)}.
\end{equation}

Similarly, let
\begin{equation}
\mathit{Parkqt}_{a,b}^{(r,s)}(q,t) \ses
\sum_{PF \in {\cal PF}_{a,b}^{(r,s)}} t^{\operatorname{area}(PF)} q^{\operatorname{dinv}(PF)}
\end{equation}
and
\begin{equation}
\mathit{Parkqt}_{a,b}^{(s)}(q,t) \ses
\sum_{PF \in {\cal PF}_{a,b}^{(s)}}
t^{\operatorname{area}(PF)} q^{\operatorname{dinv}(PF)}.
\end{equation}

This given, we have the following recursions.

\begin{prop} \label{recur}
If $0 \leq r \leq a$ and $0 \leq s < b$, then
\begin{align}
\mathit{Parkq}_{a,b}^{(r,s)}(q) 
&\ses
q^{(s+r)(a+b) - {s+r+1 \choose 2} - 1} 
\Big[ {s+r \atop s} \Big]_q
\sum_{h=1}^{b-s} 
\Big[ {r+h-1 \atop h} \Big]_q 
\sum_{k=0}^{a-r} 
\mathit{Parkq}_{a-r,b-s-1}^{(k,h-1)}(q) \cr
&\ses
q^{(s+r)(a+b) - {s+r+1 \choose 2} - 1} 
\Big[ {s+r \atop s} \Big]_q
\sum_{h=1}^{b-s} 
\Big[ {r+h-1 \atop h} \Big]_q 
\mathit{Parkq}_{a-r,b-s-1}^{(h-1)}(q).\cr
\end{align}
\end{prop}

\begin{prop} \label{ISrecur}
If $1 \leq a$ and $0 \leq s \leq b$, then
\begin{equation}
\mathit{Parkq}_{a,b}^{(s)}(q) \ses 
q^{{a+b \choose 2}-(b-s)-{a+b-s-1 \choose 2}} \sum_{r=1}^{a} \Big[ {s+r \atop r} \Big]_q \mathit{Parkq}_{b-s,a-1}^{(r-1)}(q).
\end{equation}
\end{prop}

The identity in Proposition $\ref{recur}$ is the $t=1/q$ specialization of an identity for $\mathit{Parkqt}_{a,b}^{(r,s)}$ proved by the first author in \cite{qara} and Proposition $\ref{ISrecur}$ is a specialization of a recursion for $\mathit{Parkqt}_{a,b}^{(s)}$ originally discovered by \cite{Schroder}. Below we reproduce a simple surjective proof of the latter which was given by the first author and Stout in \cite{k2}.

In point of fact, \cite{Schroder} proved, by a highly non-trivial sequence of manipulations, that
\begin{equation} \label{symm}
\mathit{Parkqt}_{a,b+1}^{(s)}(q,t) \ses 
\big\langle \Delta_{h_a} E_{b+1,s} \, , \, e_{a+b+1} \big\rangle
\end{equation}
where $\Delta_{h_a}$ is one of a family of Macdonald polynomial eigen-operators constructed in \cite{PosCon} and the $E_{n,k}$ are the symmetric polynomials introduced by \cite{qtCat}. The simple form given here for $\mathit{Parkq}_{a,b}^{(r,s)}(q)$ suggests that $\mathit{Parkqt}_{a,b}^{(r,s)}(q,t),$ may also be expressible in terms of symmetric functions in a manner that refines ($\ref{symm}$). If this can be carried out for general multicar sizes it should yield a significant refinement of the classical shuffle conjecture. This task is certainly worth pursuing in future work.

\section{The Symmetric Function Side} \label{symsec}

For the sake of completion, we provide a brief survey of the symmetric function tools necessary to understand Conjectures \ref{shuffleconj} and \ref{shuffleconj2}. For a more thorough introduction see \cite{Macdonald} and \cite{Lagr}. The space of symmetric polynomials will be denoted by $\Lambda$. The space of homogeneous symmetric polynomials of degree $m$ will be denotes by $\Lambda^{=n}$. We will express symmetric functions in terms of the following classic bases for $\Lambda^{=n}$ indexed by partitions of $n$:
\begin{itemize}
\item the power basis $\{ p_\mu \}_{\mu \vdash n}$
\item the homogeneous basis $\{ h_\mu \}_{\mu \vdash n}$
\item the elementary basis $\{ e_\mu \}_{\mu \vdash n}$
\item the Schur basis $\{ s_\mu \}_{\mu \vdash n}$
\end{itemize}
Here $\langle \cdot, \cdot \rangle$ denotes the usual scalar product on symmetric functions defined by
\begin{equation} \label{sp}
\langle s_\lambda, s_\mu \rangle = \chi(\lambda=\mu).
\end{equation}
For each partition $\lambda$, let $\lambda'$ denote the conjugate partition. We will also make use of the involution $\omega$ defined by $\omega s_\lambda \ses s_{\lambda'}$. Note that since $e_k = s_{(1^k)}$ and $h_k = s_{(k)}$ we have that $\omega h_k = e_k$ and, in general, $\omega h_\lambda = e_\lambda$.

If $E = E(t_1,t_2,t_3,\dots)$ is a formal Laurent series in the variables $t_1,t_2,t_3,\dots$, we define
\begin{displaymath}
p_k[E] \ses E(t_1^k, t_2^k, t_3^k, \dots).
\end{displaymath}
More generally, if $F$ is any symmetric function, it can be expressed as $F = Q(p_1,p_2,p_3,\dots)$ for some polynomial $Q$. This given, we define
\begin{displaymath}
F[E] = Q( p_1[E], p_2[E], p_3[E], \dots).
\end{displaymath}
This process is referred to as \emph{plethystic substitution}.

It will be convenient to denote a partition by its (french) Ferrers diagram as in the figure below. Given a partition $\mu$ and a cell $c \in \mu$ the parameters $a=a_\mu(c)$ and $l=l_\mu(c)$, called the \emph{arm} and \emph{leg}, give the number of cells north and east of $c$, respectively. 
\begin{figure}[H]
\centering
\includegraphics[width=2in]{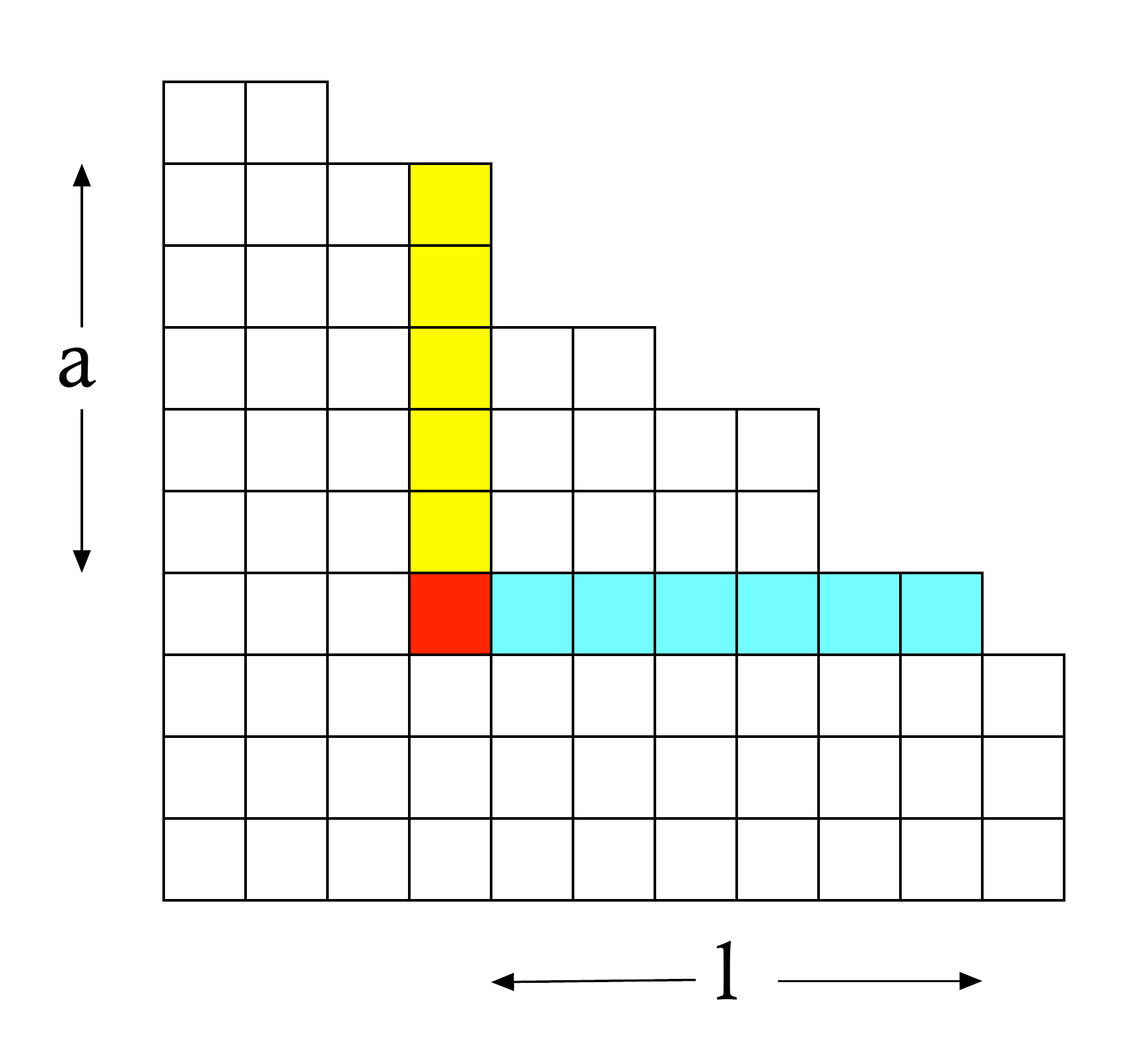}
\end{figure}
\noindent Also define
\begin{displaymath}
n(\mu) \ses \sum_{c \in \mu} l_\mu(c) \ses \sum_{i=1}^{l(\mu)} (i-1)\mu_i.
\end{displaymath}

Let
\begin{displaymath}
w_\mu(q,t) \ses \prod_{c \in \mu} (q^{a_\mu(c)} - t^{l_\mu(c)+1})(t^{l_\mu(c)}-q^{a_\mu(c)+1})
\end{displaymath}
and
\begin{displaymath}
T_\mu \ses t^{n(\mu)} q^{n(\mu')}.
\end{displaymath}

For any partition $\mu$, let $z_\mu$ be the order of the stabilizer of a permutation with cycle structure $\mu$. We have that
\begin{displaymath}
\langle \, p_\lambda, p_\mu \, \rangle \ses z_\mu \, \chi(\lambda=\mu).
\end{displaymath}
A related scalar product, called the \emph{star} scalar product, is given by
\begin{displaymath}
\langle \, p_\lambda, p_\mu \, \rangle_* \ses (-1)^{|\mu| - l(\mu)} \prod_{i} (1-t^{\mu_i})(1-q^{\mu_i}) \, z_\mu \, \chi(\lambda=\mu).
\end{displaymath}
This given, the modified Macdonald polynomials $\{ \tilde{H}_\mu(X;q,t) \}_\mu$ are the unique symmetric function basis which is upper-triangularly related to the Schur basis $\{s_\mu(x)\}_\mu$ with respect to $\langle \cdot, \cdot \rangle_*$ and satisfies the orthogonality condition
\begin{displaymath}
\big\langle \tilde{H}_\lambda(X;q,t), \tilde{H}_\mu(X;q,t) \big\rangle_* \ses w_\mu(q,t) \, \chi(\lambda=\mu).
\end{displaymath}
Following \cite{SciFi}, we let $\nabla$ denote the eigen-operator for the Macdonald polynomials with eigenvalue $T_\mu$.

For the purposes of this paper, we just need to establish that Conjecture \ref{shuffleconj} implies Conjecture \ref{shuffleconj2}. By ($\ref{GH}$), it is sufficient to show that
\begin{displaymath}
\big\langle e_n \big[ X [n+1]_q \big], h_{\mu_1} h_{\mu_2} \dots h_{\mu_k} \big\rangle \ses \prod_{i=1}^{k} q^{\mu_i \choose 2} \Big[ {n+1 \atop \mu_i } \Big]_q.
\end{displaymath}
It is well known that for any $X$ and $Y$ we have
\begin{displaymath}
h_n[XY] \ses \sum_{\mu \vdash n} s_\mu[X] s_\mu[Y].
\end{displaymath}
This expression is known as the \emph{Cauchy kernel}. When $Y=[n+1]_q$, applying $\omega$ with respect to $X$ gives
\begin{displaymath}
e_n \big[ X [n+1]_q \big] \ses 
\sum_{\mu \vdash n} s_{\mu'}[X] s_\mu \big[ [n+1]_q \big] \ses
\sum_{\mu \vdash n} s_{\mu}[X] s_{\mu'}\big[ [n+1]_q \big].
\end{displaymath}
Hence by $(\ref{sp})$, for any partition $\lambda$ we have that
\begin{displaymath}
\big\langle e_n\big[X[n+1]_q\big], s_{\lambda}[X] \big\rangle \ses s_{\lambda'}\big[ [n+1]_q \big] \ses \omega s_{\lambda} \big[ [n+1]_q \big].
\end{displaymath}
Since $\{ s_\mu \}_{\mu \vdash n}$ is a basis for $\Lambda^{=n}$, it follows that for any $P \in \Lambda^{=n}$ we have that
\begin{displaymath}
\big\langle \, e_n\big[X[n+1]_q\big], P[X] \, \big\rangle \ses \omega P \big[ [n+1]_q \big].
\end{displaymath}
In particular,
\begin{displaymath}
\big\langle e_n \big[ X [n+1]_q \big], h_{\mu_1} h_{\mu_2} \dots h_{\mu_k} \big\rangle \ses \prod_{i=1}^{k} e_{\mu_i}\big[ [n+1]_q \big].
\end{displaymath}
Then noticing that
\begin{displaymath}
e_a[1 + q + \dots + q^n] \ses q^{a \choose 2} \Big[ {n+1 \atop a} \Big]_q
\end{displaymath}
completes the proof that Conjecture \ref{shuffleconj} implies Conjecture \ref{shuffleconj2}.

\section{Key Identities}

Surprisingly, our results depend only on Propositions $\ref{recur}$ and $\ref{ISrecur}$ and the following simple $q$-binomial identity.

\begin{lemma}\label{qbin}
If $1 \leq m \leq n$ and $k \geq 0$, then
\begin{displaymath}
\Big[ {n+k \atop n} \Big]_q 
\ses 
\sum_{j=0}^k q^{m(k-j)}
\Big[ {m + j - 1 \atop m-1} \Big]_q
\Big[ {n-m + k-j \atop n-m} \Big]_q.
\end{displaymath}
\end{lemma}

\begin{proof}

Let $R(n,k)$ be the set of paths in the $n \times k$ rectangle from the southwest corner $(0,0)$ to the northeast corner $(n,k)$ with $k$ north edges and $n$ east edges. For such a path $\Pi$ let $\operatorname{area}(P)$ denote the area above $\Pi$. Recall that
\begin{displaymath}
\Big[ {n+k \atop k} \Big]_q \ses \sum_{\Pi \in R(n,k)} q^{\operatorname{area}(\Pi)}.
\end{displaymath}

For a fixed $m$ and for a given path $\Pi \in R(n,k)$, let $j$ be the height of the $m^{th}$ east step of $\Pi$. Note that $\Pi$ consists of a path from $(0,0)$ to $(m-1,j)$, an east step, and a path from $(m,j)$ to $(n,k)$. 

Consider the example below.
\begin{figure}[H]
\centering
\includegraphics[width=2.5in]{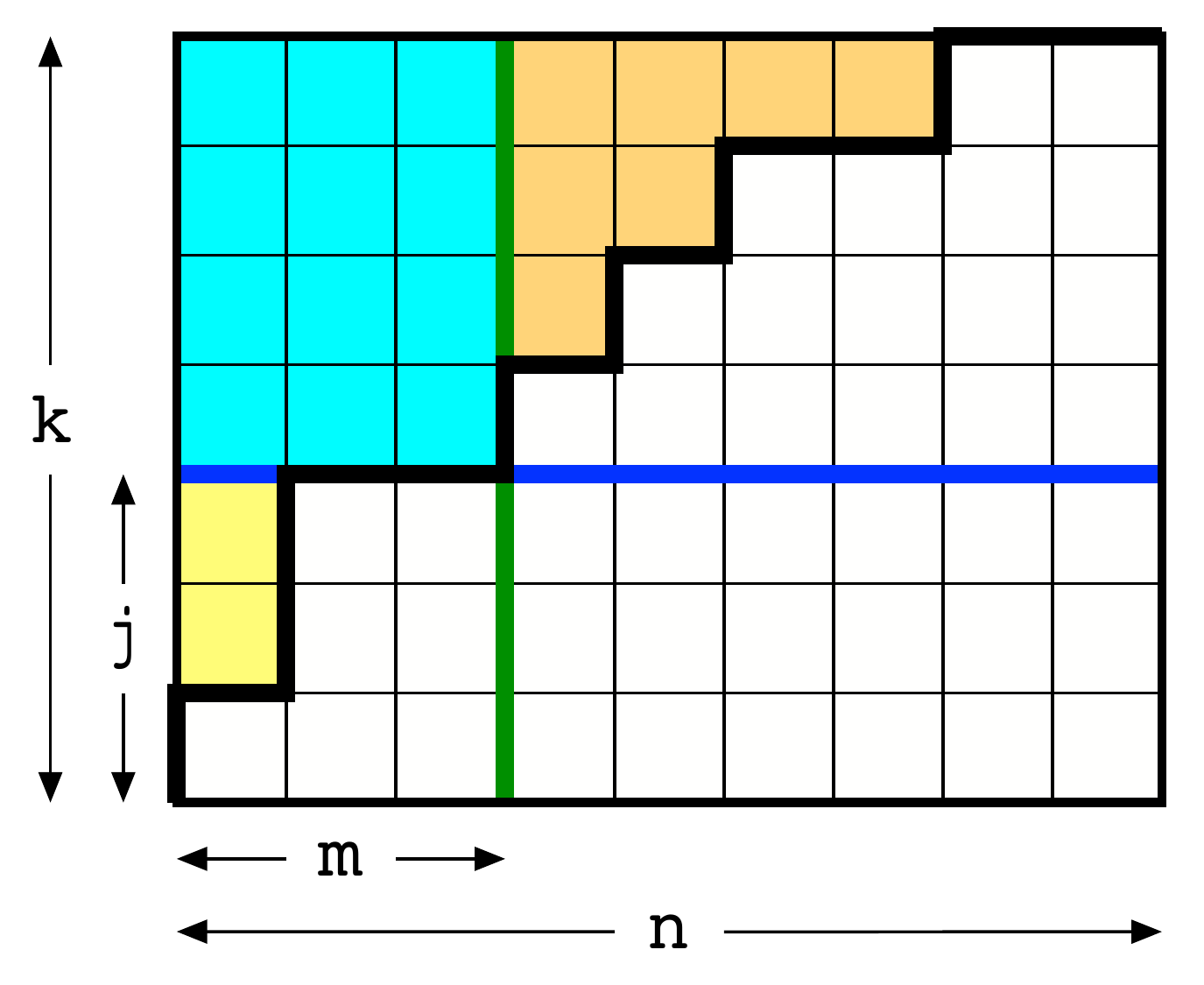}
\end{figure}

Any path of $R(m-1,j)$ and any path of $R(n-m,k-j)$ can be combined in this way to give a path $\Pi \in R(n,k)$. Conversely
any path $\Pi \in R(n,k)$ has a unique $j$ which splits the path at the $m^{th}$ east step. Hence the desired identity follows from the fact that the area of $\Pi$ is equal to $m(k-j)$ plus the area of the path from $(0,0)$ to $(m-1,j)$ and the area of the path from $(m,j)$ to $(n,k)$ (the blue, yellow, and orange sections above, respectively).
\end{proof}

For the sake of completeness, we provide sketches of the proofs of Propositions $\ref{recur}$ and $\ref{ISrecur}$ in the present notation. These proofs involve similar manipulations of Parking Functions and both rely on the following well-known fact.

Let $W(1^a 2^b)$ be the set of words consisting of $a$ 1's and $b$ 2's. Let $w_i$ denote the $i^{th}$ letter of $w$. For $w \in W(1^a 2^b)$, set $\operatorname{inv}(w) = \sum_{i<j} \chi( w_i > w_j )$ and $\operatorname{coinv}(w) = \sum_{i < j} \chi( w_i < w_j )$. Then for any $n,m \geq 0$, we have that
\begin{equation} \label{inv}
\Big[ {n+m \atop m} \Big]_q \ses \sum_{w \in W(1^n 2^m)} q^{\operatorname{inv}(w)} \ses \sum_{w \in W(1^m 2^n)} q^{\operatorname{coinv}(w)}.
\end{equation}

To deal with Parking Functions whose reading word is a shuffle
of $a$ small cars and $b$ big cars it is convenient to depict them as tableaux obtained by replacing each small car by a ``$1$'' and each big car by a ``$2$''. 
\begin{figure}[H]
\centering
\includegraphics[width=3.5in]{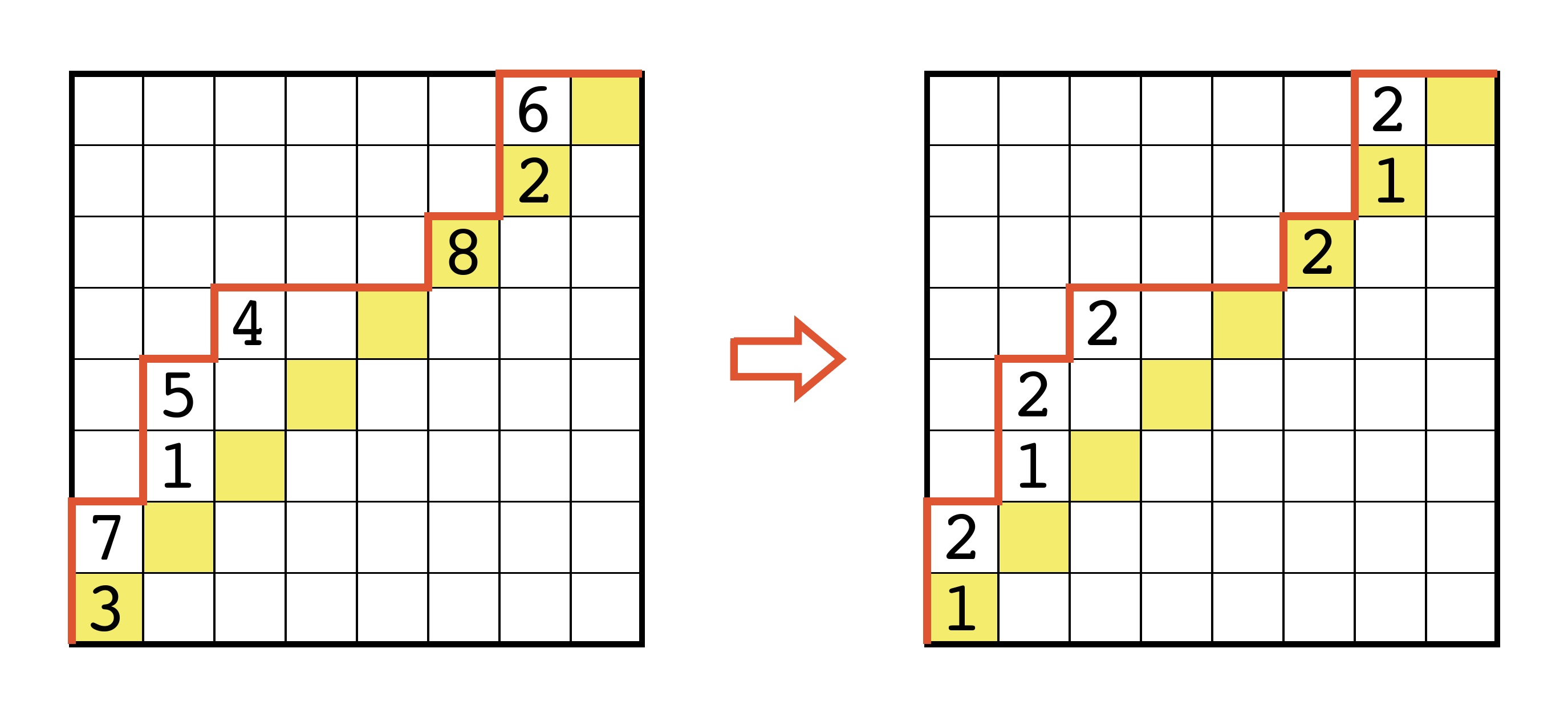}
\end{figure}
This certainly does not affect the area statistic. To show that this replacement does not affect the dinv formula, we only need to point out that the shuffle condition assures that pairs of cars of the same size will never contribute a diagonal inversion. To pass from a $1,2$-tableau to the original Parking Function we simply replace all the $1$'s successively from right to left and by decreasing area numbers with the letters $1,2,\ldots ,a$ and likewise all $2's$ in the same succession with the letters $a+1,a+2,\ldots , a+b$.

Let us recall that we defined
\begin{displaymath}
\mathit{Parkqt}_{a,b}^{(r,s)}(q,t) \ses
\sum_{PF \in {\cal PF}_{a,b}^{(r,s)}} t^{\operatorname{area}(PF)} q^{\operatorname{dinv}(PF)}
\end{displaymath}
and
\begin{displaymath}
\mathit{Parkqt}_{a,b}^{(s)}(q,t) \ses
\sum_{PF \in {\cal PF}_{a,b}^{(s)}}
t^{\operatorname{area}(PF)} q^{\operatorname{dinv}(PF)}
\end{displaymath}
so that 
\begin{displaymath}
q^{a+b \choose 2} \mathit{Parkqt}_{a,b}^{(r,s)}(q,1/q) = \mathit{Parkq}_{a,b}^{(r,s)}(q) \hbox{ \hskip 6pt and \hskip 6pt } q^{a+b \choose 2} \mathit{Parkqt}_{a,b}^{(s)}(q,1/q) = \mathit{Parkq}_{a,b}^{(s)}(q).
\end{displaymath} 
Propositions $\ref{recur}$ and $\ref{ISrecur}$ can be obtained from the following two results by setting $t=1/q$ and multiplying by $q^{a+b \choose 2}$.

\begin{prop}\label{trecur}
If $0 \leq r \leq a$ and $0 \leq s < b$, then
\begin{equation} \label{tr}
\mathit{Parkqt}_{a,b}^{(r,s)}(q,t) = t^{a+b-r-s} \Big[ {s+r \atop s} \Big]_q \sum_{h=1}^{b-s} \Big[ {r + h - 1 \atop h} \Big]_q \sum_{k=0}^{a-r} \mathit{Parkqt}_{a-r,b-s-1}^{(k,h-1)}(q,t).
\end{equation}
\end{prop}

\begin{proof}
Let $0 \leq r \leq a$, $0 \leq s < b$, $1 \leq h \leq b-s$ and $0 \leq k \leq a-r$. Let $PF \in {\cal PF}_{a-r,b-s-1}^{(k,h-1)}$. We begin by adding a car of size $2$ to the main diagonal to obtain $PF' \in {\cal PF}_{a-r,b-s}^{(k,h)}$. We have $\operatorname{area}(PF)=\operatorname{area}(PF')$ and $\operatorname{dinv}(PF)=\operatorname{dinv}(PF')$.

\begin{figure}[H]
\centering
\includegraphics[width=3.2in]{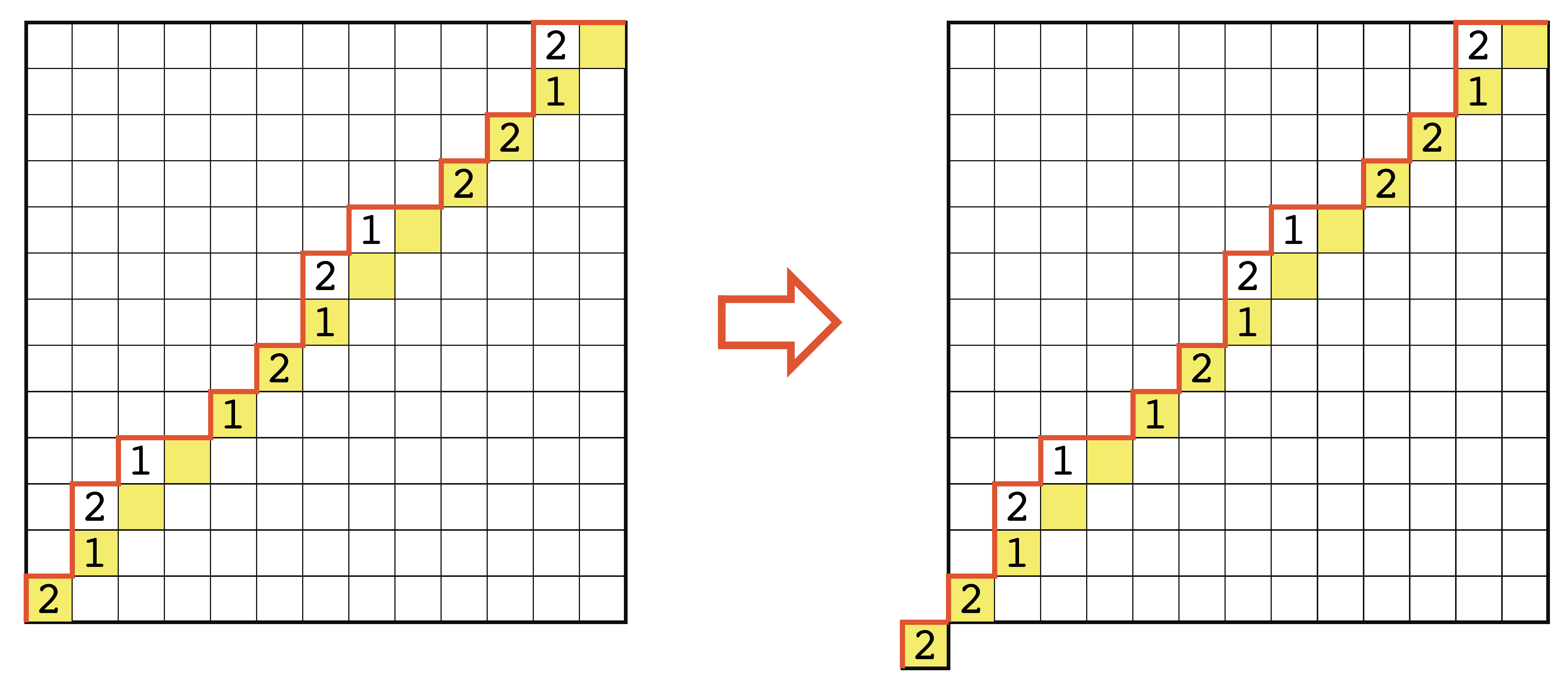}
\end{figure}

Now split up $PF'$ into $h$ blocks beginning with each car of size $2$ on the main diagonal. Next we will construct an element of ${\cal PF}_{a,b-s}^{(r,0)}$ for each word $w$ consisting of $r$ 1's and $h$ 2's which begins with a $1$. We do this by placing each $1$ we encounter within $w$ in the next available spot on the main diagonal. When we encounter a $2$ in $w$, we insert the next available block of $PF'$ along the first diagonal directly on top of the $1$ preceding this $2$ in $w$ if there is one. Let this diagonal be known as the $1$-diagonal. If there is not a $1$ preceding this $2$ in $w$, place the corresponding block along the $1$-diagonal directly after the previous block.

\begin{figure}[H]
\centering
\includegraphics[width=4.3in]{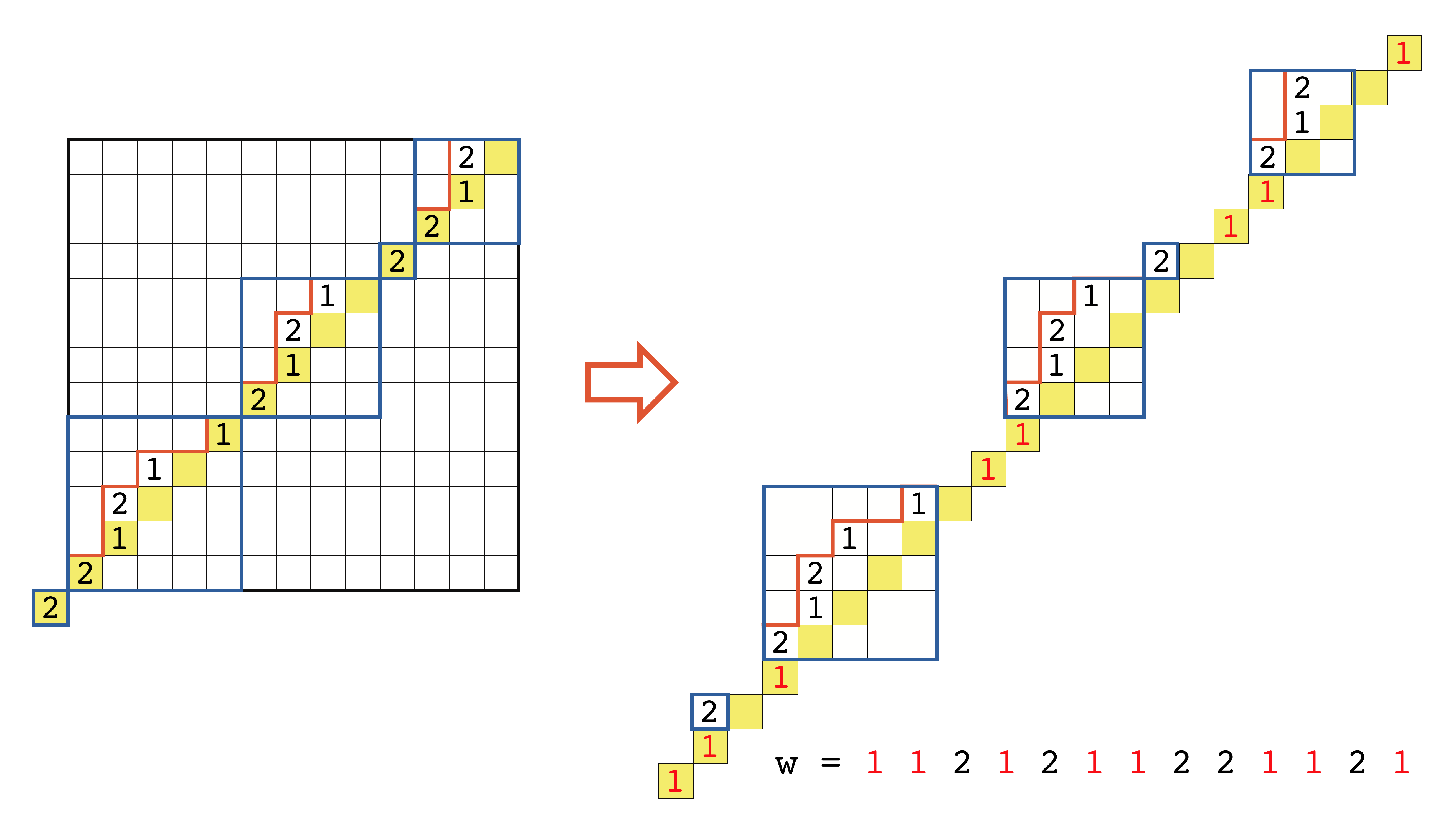}
\end{figure}

Since each of the $a+b-r-s$ cars of $PF'$ have been shifted up one diagonal and all new cars have been added to the main diagonal, the area of the resulting Parking Function is $a+b-r-s + \operatorname{area}(PF)$. Moreover, we have created no new primary dinv and the additional secondary dinv is equal to the number of inversions in $w$, $\operatorname{inv}(w)$.

Finally, we construct an element of ${\cal PF}_{a,b}^{(r,s)}$ for each word $v$ consisting of $r$ 1's and $s$ 2's. This is done by inserting $1$'s into the main diagonal of the previous Parking Function so that reading the main diagonal from left to right gives $v$.

\begin{figure}[H]
\centering
\includegraphics[width=4.8in]{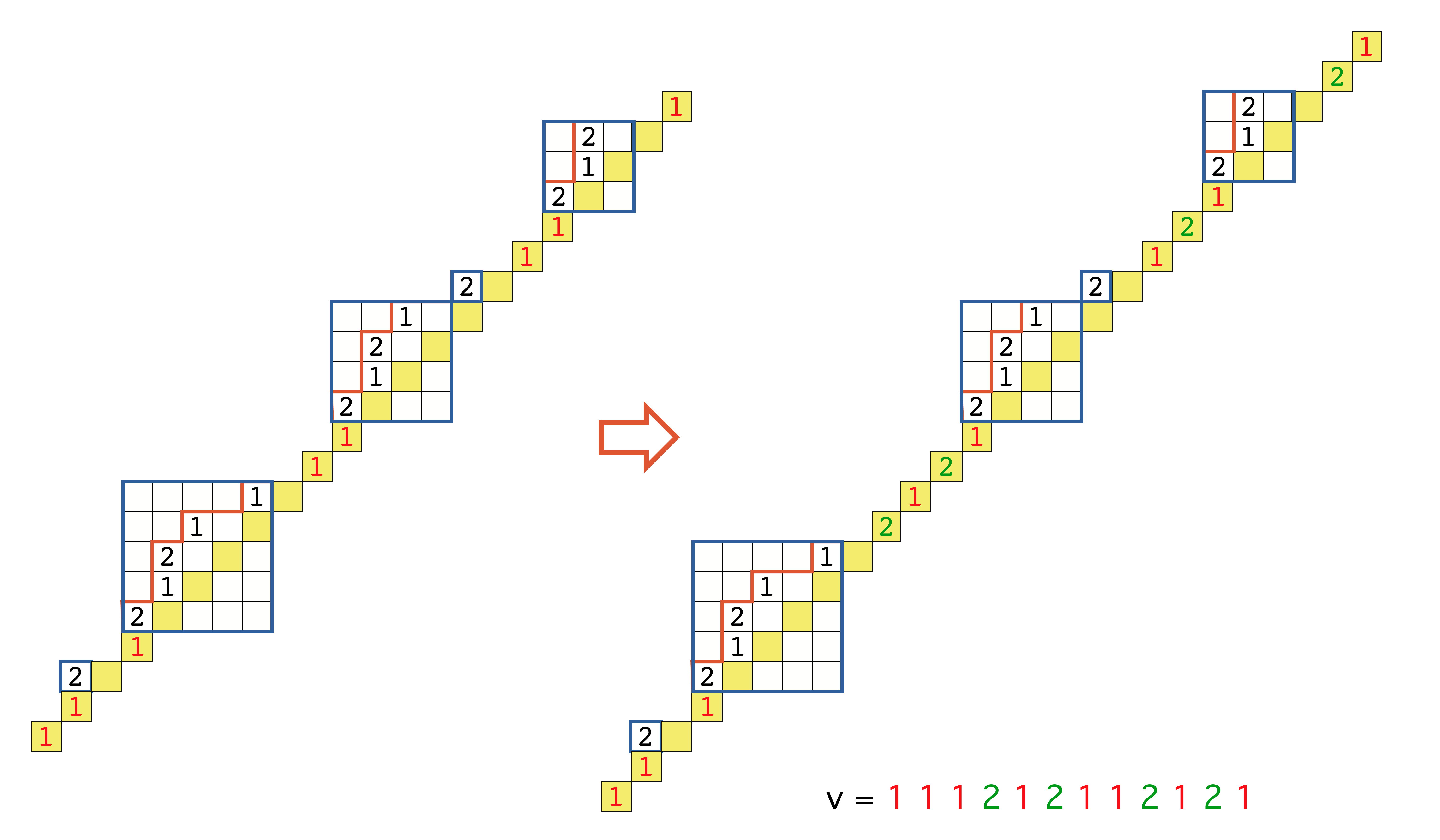}
\end{figure}

Each element of ${\cal PF}_{a,b}^{(r,s)}$ is uniquely created in this way by taking $k$ to be the number of 1's on the 1-diagonal, $h-1$ to be the number of 2's on the 1-diagonal, and by an appropriate choice of $PF$, $w$ and $v$. Therefore $(\ref{tr})$ follows from the observations we have made above regarding the change in area and dinv combined with ($\ref{inv}$).
\end{proof}

\begin{prop} \label{tISrecur}
If $1 \leq a$ and $0 \leq s \leq b$, then
\begin{equation}
\mathit{Parkqt}_{a,b}^{(s)}(q,t) \ses 
t^{b-s} \sum_{r=1}^{a} \Big[ {s+r \atop r} \Big]_q \mathit{Parkqt}_{b-s,a-1}^{(r-1)}(q,t).
\end{equation}
\end{prop}

\begin{proof}

Let $1 \leq r \leq a$ and $0 \leq s \leq b$. Consider any $PF \in {\cal PF}_{b-s, a-1}^{(r-1)}$ and any $w \in W(1^r 2^s)$. Begin by inserting a $2$ into the southwest corner of $PF$ to get $PF' \in {\cal PF}_{b-s,a}^{(r)}$. As before this leaves the area and dinv unchanged.

\begin{figure}[H]
\centering
\includegraphics[width=3.3in]{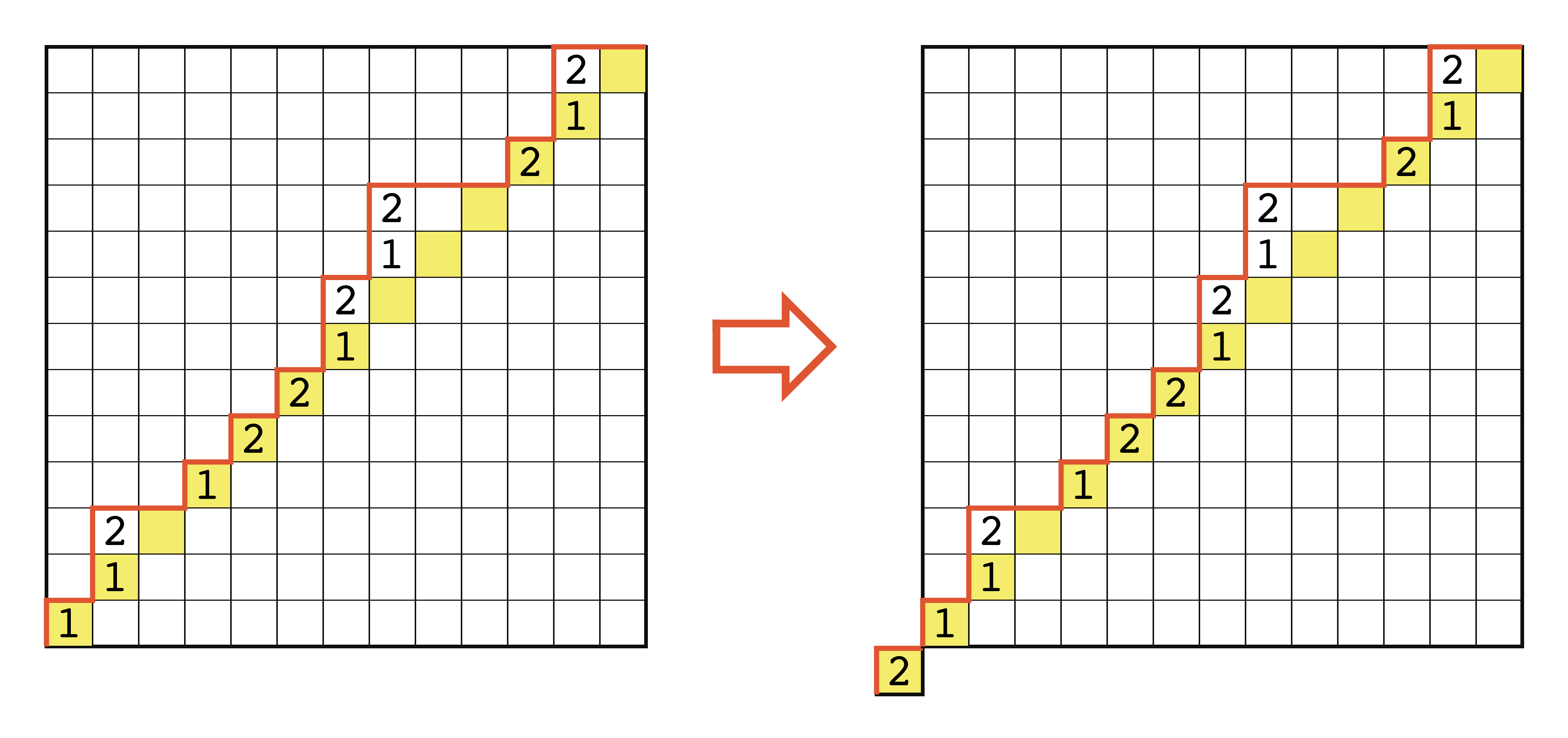}
\end{figure}

Now shift each $1$ into the cell immediately west of it's current position. Note that the result does not represent a Parking Function and that we have changed the supporting Dyck path. Then replace every $1$ with a $2$ and vice versa (i.e. turn all small cars into big cars and big cars into small cars). This gives a Parking Function $PF'' \in {\cal PF}_{a,b-s}^{(0)}$. Note that $\operatorname{area}(PF'') = \operatorname{area}(PF') + b-s$ since each $1$ of $PF'$ which was shifted left increased the area.

\begin{figure}[H]
\centering
\includegraphics[width=5.2in]{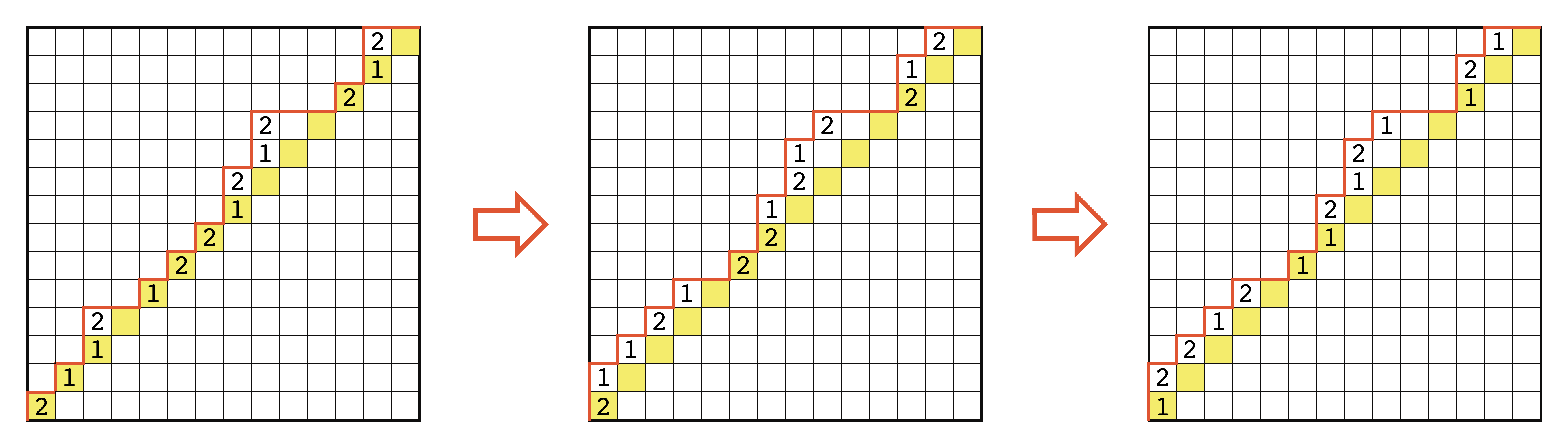}
\end{figure}

Furthermore, $\operatorname{dinv}(PF'') = \operatorname{dinv}(PF')$. This is because the pairs of $1$'s and $2$'s which cause primary dinv in $PF'$ are precisely those which cause secondary dinv in $PF''$.
\begin{figure}[H]
\centering
\includegraphics[width=1.2in]{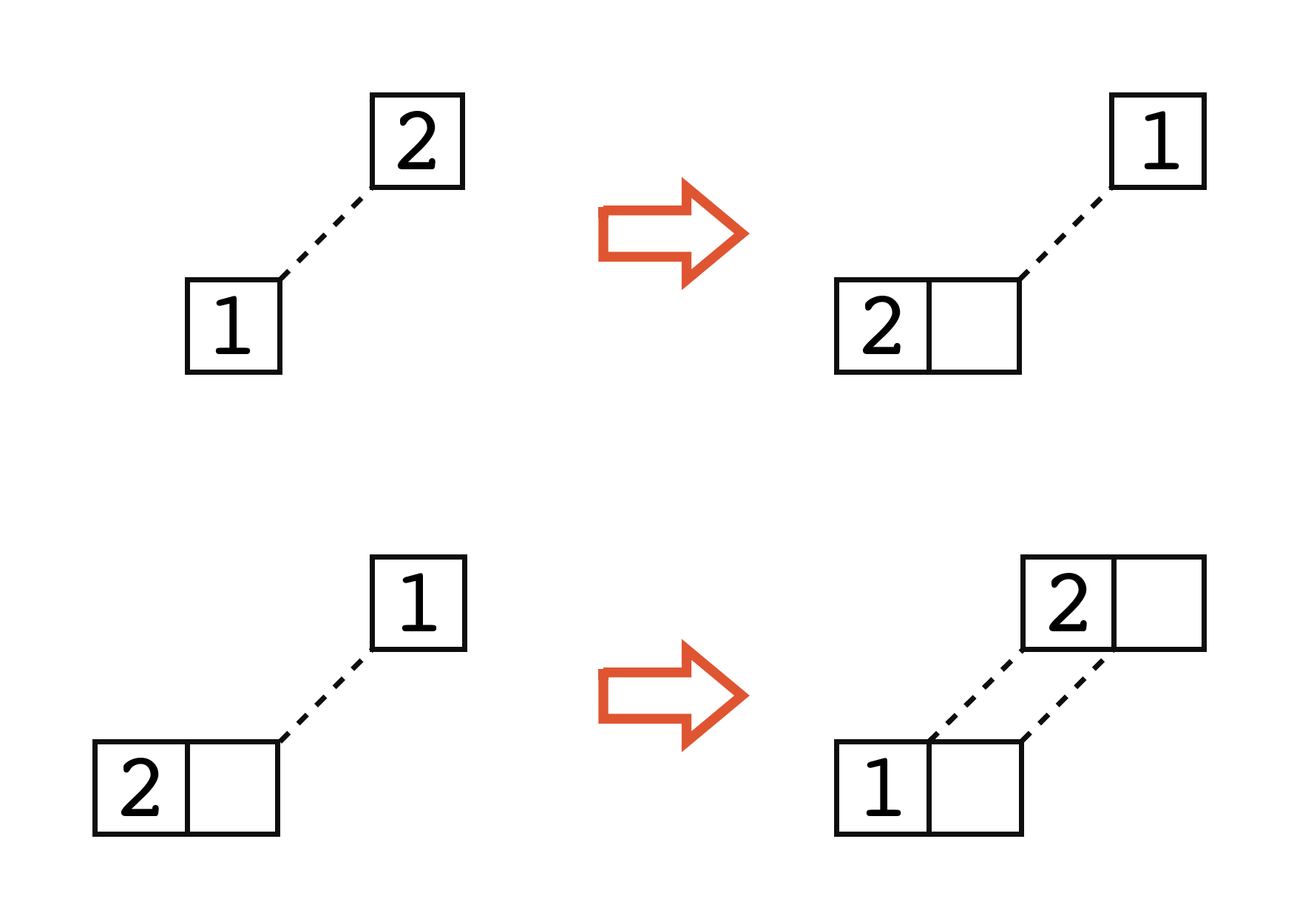}
\end{figure}
Similarly, the pairs of $1$'s and $2$'s which cause secondary dinv in $PF'$ are those which cause primary dinv in $PF''$.

Now break $PF''$ into $a$ blocks starting with each $1$ on the main diagonal. Then insert $2$s into the main diagonal according to $w$ as we did in the proof of Proposition $\ref{trecur}$.

\begin{figure}[H]
\centering
\includegraphics[width=4in]{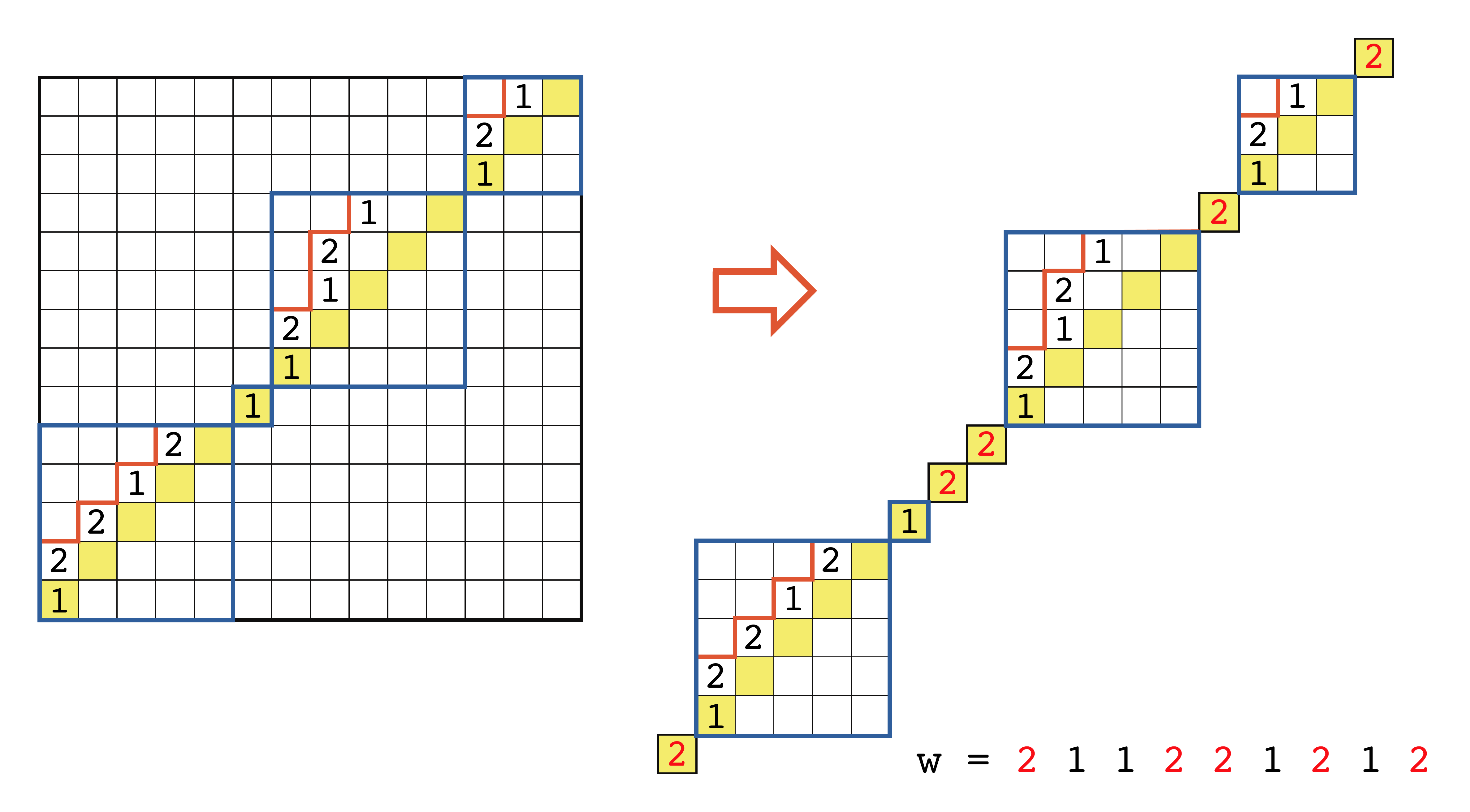}
\end{figure}

The result is a Parking Function $PF^*$ in ${\cal PF}_{a,b}^{(s)}$. Clearly $\operatorname{area}(PF^*) = \operatorname{area}(PF'') = \operatorname{area}(PF) + b-s$. Note also that $\operatorname{dinv}(PF^*) = \operatorname{dinv}(PF'') + \operatorname{coinv}(w) = \operatorname{dinv}(PF) + \operatorname{coinv}(w)$.

Furthermore, we can see that this operation is reversible and surjective. Hence summing over all possible choices of $PF$, $r$ and $w$ gives the desired result.
\end{proof}

\section{$q$-Binomial Formulas} \label{comp}

In order to arrive at the formula for $\mathit{Parkq}_{a,b}^{(r,s)}$ given in Theorem $\ref{qara}$, we have the following intermediate step.

\begin{thm}\label{ISthm}
For all $a >0$, and $0 \leq s \leq b$, we have
\begin{equation} \label{InnerSum}
\mathit{Parkq}_{a,b}^{(s)}(q) 
\ses 
q^{{a+b \choose 2} - (b-s)a} 
\Big[{a+b \atop a}\Big]_q 
\Big[{a+b-s-1 \atop a-1}\Big]_q 
\frac{[s+1]_q}{[b+1]_q}.
\end{equation}

When $a=0$, we have
\begin{equation} \label{0InnerSum}
\mathit{Parkq}_{0,b}^{(s)}(q) 
\ses 
q^{b \choose 2}\chi(b=s).
\end{equation}
\end{thm}

\begin{proof}
When $a=0$, the only Parking Function with $a$ small cars and $b$ large cars is the one which has all large cars on the diagonal, i.e. when $b=s$. That single Parking Function has $\operatorname{dinv}=0$ and $\operatorname{coarea} = {b \choose 2}$. Hence ($\ref{0InnerSum}$) holds for all $0 \leq s \leq b$.

Similarly, if $b=s$ then all large cars, and consequently all small cars, are on the main diagonal. Such a Parking Function $P$ has $\operatorname{coarea}(P) = {a+b \choose 2}$ and $\operatorname{dinv}(P) = \operatorname{inv}( \sigma(P) )$, since all dinv occurs as primary dinv on the main diagonal. Hence
\begin{equation} \label{b=s}
\mathit{Parkq}_{a,b}^{(b)} = q^{a+b \choose 2} \Big[ {a + b \atop a} \Big]_q
\end{equation}
which is the desired specialization of $(\ref{InnerSum})$.

Given these two cases, we will induct on $\max\{a,b\}$. Suppose $a>0$, $b>s$ and that the claim holds for all smaller cases. Then by induction and Proposition $\ref{ISrecur}$, we have
\begin{align*}
\mathit{Parkq}_{a,b}^{(s)}(q) &\ses q^{{a+b \choose 2} - (b-s) - {a+b-s-1 \choose 2}} \sum_{r} \Big[ {s+r \atop r} \Big]_q \mathit{Parkq}_{b-s,a-1}^{(r-1)}(q) \cr
&\ses q^{{a+b \choose 2} - (b-s) - {a+b-s-1 \choose 2}} \sum_{r} \Big[ {s+r \atop r} \Big]_q \cr
& \hskip 60 pt \times q^{{a+b-s-1 \choose 2} - (a-r)(b-s)} \Big[ {a+b-s-1 \atop b-s} \Big]_q \Big[ {a-r+b-s-1 \atop b-s-1} \Big]_q \frac{[r]_q}{[a]_q} \cr
&\ses q^{{a+b \choose 2}-(b-s)a} \Big[ {a+b-s-1 \atop a-1} \Big]_q \frac{[s+1]_q}{[a]_q} \\&\hskip 60 pt\times\sum_{r} q^{(b-s)(r-1)} \Big[ {s+r \atop s+1} \Big]_q \Big[ {a-r+b-s-1 \atop b-s-1} \Big]_q.
\end{align*}

Applying Lemma $\ref{qbin}$ to the sum above with $n=b+1$, $k=a-1$, $m=b-s$ and $j=a-r$ gives
\begin{align*}
\mathit{Parkq}_{a,b}^{(s)}(q) &\ses q^{{a+b \choose 2}-(b-s)a} \Big[ {a+b-s-1 \atop a-1} \Big]_q \Big[ {a+b \atop b+1} \Big]_q \frac{[s+1]_q}{[a]_q} \cr
&\ses q^{{a+b \choose 2}-(b-s)a} \Big[ {a+b-s-1 \atop a-1} \Big]_q \Big[ {a+b \atop a} \Big]_q \frac{[s+1]_q}{[b+1]_q}. \cr
\end{align*}
By induction, the claim holds for all $0 \leq a$ and $0 \leq s \leq b$.
\end{proof}

Now we have the tools to prove our refinement of the Shuffle Conjecture for $k=2$, $t=1/q$. Recall
\begin{repthm}{qara}
For all $0 \leq s < b$ and $0 \leq r \leq a$ we have 
\begin{displaymath}
\mathit{Parkq}_{a,b}^{(r,s)}(q) \ses q^{{a+b \choose 2}-(a-r+1)(b-s)} 
\Big[ {a+b-s-1 \atop a} \Big]_q 
\Big[ {a-r+b-s \atop a-r} \Big]_q 
\Big[ {r+s \atop s} \Big]_q 
{[r]_q\over [a-r+b-s]_q}
\end{displaymath}
and for all $0\leq b$, $0 \leq r \leq a$ we have
\begin{displaymath}
\mathit{Parkq}_{a,b}^{(r,b)}(q)
 \ses 
\chi(a=r) \, q^{a+b \choose 2} \Big[ {a+b \atop a} \Big]_q.
\end{displaymath}
\end{repthm}

\begin{proof}

Note that when all large cars are on the main diagonal (i.e. when $s=b$) all small cars must be on the diagonal as well. Hence if $r <a$, $\mathit{Parkq}_{a,b}^{(r,b)} = 0$. When $r=a$, we have that\begin{displaymath}
\mathit{Parkq}_{a,b}^{(a,b)} 
\ses 
q^{a+b \choose 2} 
\Big[ {a+b \atop a} \Big]_q
\end{displaymath}
by an observation similar to that of ($\ref{b=s}$). Note that this is not a special case of Theorem $\ref{qara}$ as we assumed that $s<b$.

We will break the proof of Theorem $\ref{qara}$ into two cases depending on whether $r=a$ or $r<a$.

When $r=a$, $(\ref{0InnerSum})$ and Proposition $\ref{recur}$ give
\begin{align*}
\mathit{Parkq}_{a,b}^{(a,s)}(q) 
&\ses 
q^{(s+a)(a+b) - {s+a+1 \choose 2} - 1} \Big[ {s+r \atop s} \Big]_q \sum_{h=1}^{b-s} \Big[ {a+h-1 \atop h} \Big]_q \mathit{Parkq}_{0,b-s-1}^{(h-1)}(q) \cr
&\ses 
q^{(s+a)(a+b) - {s+a+1 \choose 2} - 1} 
\Big[ {s+a \atop s} \Big]_q 
\sum_{h=1}^{b-s} 
\Big[ {a+h-1 \atop h} \Big]_q 
q^{b-s-1 \choose 2} 
\chi(b-s-1 = h-1) \cr
&\ses 
q^{{a+b \choose 2} - (b-s)} 
\Big[ {s+a \atop s} \Big]_q 
\Big[ {a+b-s-1 \atop b-s} \Big]_q
\end{align*}
and this is the desired specialization of Theorem $\ref{qara}$. Note that the last equality is simply due to the fact that
\begin{displaymath}
(s+a)(a+b)-{s+a+1 \choose 2}-1+ {b-s-1 \choose 2} 
\ses
{a+b \choose 2} - (b-s)
\end{displaymath}
which can be directly verified by expanding both sides.

Now suppose that $r<a$. Then by $(\ref{InnerSum})$ and Proposition $\ref{recur}$
\begin{align*} \label{a>r}
\mathit{Parkq}_{a,b}^{(r,s)}(q)
&\ses
q^{(s+r)(a+b) - {s+r+1 \choose 2}-1} \Big[ {s+r \atop s} \Big]_q
\sum_{h} \Big[ {r+h-1 \atop h} \Big]_q \mathit{Parkq}_{a-r,b-s-1}^{(h-1)}(q) \cr
&\ses
q^{(s+r)(a+b) - {s+r+1 \choose 2}-1} \Big[ {s+r \atop s} \Big]_q
\sum_{h} \Big[ {r+h-1 \atop h} \Big]_q q^{{a-r+b-s-1 \choose 2}-(b-s-h)(a-r)}\cr
& \hskip 40pt \times  \Big[ {a-r+b-s-1 \atop a-r} \Big]_q \Big[ {a-r + b-s - h - 1 \atop a-r-1} \Big]_q \frac{ [h]_q}{ [b-s]_q} \cr
&\ses 
q^{{a+b \choose 2} - (a-r+1)(b-s)} \Big[ {s+r \atop s} \Big]_q \Big[ {a-r+b-s \atop a-r} \Big]_q \frac{[r]_q}{[a-r+b-s]_q} \cr
&\hskip 40pt \times \sum_{h} q^{(h-1)(a-r)} \Big[ {r+h-1 \atop h-1} \Big]_q \Big[ {a-r+b-s-h-1 \atop a-r-1} \Big]_q.
\end{align*}

Applying Lemma $\ref{qbin}$ with $n=a$, $k=b-s-1$, $m=r+1$ and $j=h-1$ gives
\begin{displaymath}
\mathit{Parkq}_{a,b}^{(r,s)}(q)
\ses
q^{{a+b \choose 2} - (a-r+1)(b-s)} \Big[ {s+r \atop s} \Big]_q \Big[ {a-r+b-s \atop a-r} \Big]_q \frac{[r]_q}{[a-r+b-s]_q} \Big[ {a+b-s-1 \atop a} \Big]_q
\end{displaymath}
as desired, and our proof of Theorem 1.1 is now complete.
\end{proof}

Note that the result of summing over $r$ in the formulas of Theorem $\ref{qara}$ agrees with the formulas of Theorem $\ref{ISthm}$ by yet another application of Lemma $\ref{qbin}$. This computation is left for the reader. However, we cannot use these methods to find a nice closed form for $\sum_{s=0}^b \mathit{Parkq}_{a,b}^{(r,s)}$. Computer experimentation reveals that this polynomial is not even necessarily a ratio of $q$-analogs and powers of $q$.

It remains to show that Theorem $\ref{qara}$ is indeed a refinement of the Shuffle conjecture in the case $k=2$ and $t=1/q$. Conveniently, this can be accomplished using Lemma $\ref{qbin}$ one last time.

\begin{repthm}{wolf}
For $a \geq 0$ and $b \geq 0$,
\begin{equation} \label{wolfeq}
\sum_{s} \mathit{Parkq}_{a,b}^{(s)} \ses
\frac{q^{{a\choose 2}+{b\choose 2}} }{ [a+b+1]_q}
\Big[ {a+b+1\atop a} \Big]_q
\Big[ {a+b+1\atop b} \Big]_q.
\end{equation}
\end{repthm}

\begin{proof}
When $a=0$, we have already observed that the only possible Parking Function occurs when $b=s$ so that
\begin{displaymath}
\sum_{s} \mathit{Parkq}_{0,b}^{(s)} \ses q^{b \choose 2}
\end{displaymath}
which is the desired specialization of $(\ref{wolfeq})$.

Suppose that $a > 0$ and $b \geq 0$. Then by $(\ref{InnerSum})$ we have
\begin{align*}
\sum_{s} \mathit{Parkq}_{a,b}^{(s)} 
&\ses 
\sum_{s} q^{ {a+b \choose 2} - (b-s)a} \Big[ {a+b \atop a} \Big]_q \Big[ {a+b-s-1 \atop a-1} \Big]_q \frac{[s+1]_q}{[b+1]_q} \cr
&\ses
\frac{ q^{{a+b \choose 2}-ab} }{[a+b+1]_q} \Big[ {a+b+1 \atop a} \Big] \sum_{s} q^{as} \Big[ {a+b-s-1 \atop a-1} \Big]_q \Big[ {s+1 \atop 1} \Big]_q. \cr
\end{align*}

Now applying Lemma $\ref{qbin}$ with $n=a+1$, $m=a$, $k=b$ and $j=b-s$ gives
\begin{align*}
\sum_{s} \mathit{Parkq}_{a,b}^{(s)} 
&\ses \frac{ q^{{a+b \choose 2}-ab} }{[a+b+1]_q} \Big[ {a+b+1 \atop a} \Big] \Big[ {a+b+1 \atop a+1} \Big]_q \cr
&\ses \frac{ q^{{a \choose 2} + {b \choose 2}}}{[a+b+1]_q} \Big[ {a+b+1 \atop a}\Big]_q \Big[ {a+b+1 \atop b} \Big]_q. \cr
\end{align*}
which completes our proof.
\end{proof}

Note that in particular, this formula gives a q-analogue of the previously well-studied Narayana numbers. In particular it gives an explicit formula for the specialization of a q,t- analogue of the Narayana numbers introduced in \cite{nara} when $t=1/q$.
 
\bibliographystyle{abbrvnat}
\bibliography{k2ShuffleRefined}

\begin{thebibliography}{12}
\providecommand{\natexlab}[1]{#1}
\providecommand{\url}[1]{\texttt{#1}}
\expandafter\ifx\csname urlstyle\endcsname\relax
  \providecommand{\doi}[1]{doi: #1}\else
  \providecommand{\doi}{doi: \begingroup \urlstyle{rm}\Url}\fi

\bibitem[{Aval} et~al.(2013){Aval}, {D'Adderio}, {Dukes}, {Hicks}, and {Le
  Borgne}]{qara}
J.-C. {Aval}, M.~{D'Adderio}, M.~{Dukes}, A.~{Hicks}, and Y.~{Le Borgne}.
\newblock {Statistics on parallelogram polyominoes and a q,t-analogue of the
  Narayana numbers}.
\newblock \emph{ArXiv e-prints}, Jan. 2013.

\bibitem[Bergeron and Garsia(1999)]{SciFi}
F.~Bergeron and A.~M. Garsia.
\newblock Science fiction and macdonald's polynomials.
\newblock \emph{CRM Proceedings \& Lecture Notes, American Mathematical
  Society}, 22:\penalty0 1--52, 1999.

\bibitem[Bergeron et~al.(1999)Bergeron, Garsia, and Tesler]{PosCon}
F.~Bergeron, A.~M. Garsia, and G.~Tesler.
\newblock Identities and positivity conjectures for some remarkable operators
  in the theory of symmetric functions.
\newblock \emph{Methods in Appl. Anal.}, 6:\penalty0 363--420, 1999.

\bibitem[Dukes and Borgne(2013)]{nara}
M.~Dukes and Y.~L. Borgne.
\newblock Parallelogram polyominoes, the sandpile model on a complete bipartite
  graph, and a -narayana polynomial.
\newblock \emph{Journal of Combinatorial Theory, Series A}, 120\penalty0
  (4):\penalty0 816 -- 842, 2013.
\newblock ISSN 0097-3165.
\newblock \doi{10.1016/j.jcta.2013.01.004}.
\newblock URL
  \url{http://www.sciencedirect.com/science/article/pii/S0097316513000150}.

\bibitem[Garsia et~al.(2011)Garsia, Hicks, and Stout]{k2}
A.~Garsia, A.~Hicks, and A.~Stout.
\newblock The case $k=2$ of the shuffle conjecture.
\newblock \emph{J. Comb.}, 2\penalty0 (2):\penalty0 193--229, 2011.

\bibitem[Garsia and Haglund(2002)]{qtCat}
A.~M. Garsia and J.~Haglund.
\newblock A proof of the q,t-catalan positivity conjecture.
\newblock \emph{Discrete Math.}, 256:\penalty0 677--717, 2002.

\bibitem[Garsia and Haiman(1996)]{GH}
A.~M. Garsia and M.~Haiman.
\newblock A remarkable q,t-catalan sequence and q-lagrange inversion.
\newblock \emph{J. Algebraic Combin.}, 5:\penalty0 191--244, 1996.

\bibitem[Garsia et~al.(1999)Garsia, Haiman, and Tesler]{Lagr}
A.~M. Garsia, M.~Haiman, and G.~Tesler.
\newblock Explicit plethystic formulas for the madonald q,t-kostka
  coefficients.
\newblock \emph{S\'eminaire Lotharingien de Combinatoire}, B42m:\penalty0 45
  pp., 1999.

\bibitem[Haglund(2004)]{Schroder}
J.~Haglund.
\newblock A proof of the q,t-schr\"oder conjecture.
\newblock \emph{Internat. Math. Res. Notices}, 11:\penalty0 525--560, 2004.

\bibitem[Haglund et~al.(2005)Haglund, Haiman, Loehr, Remmel, and
  Ulyanov]{Shuffle}
J.~Haglund, M.~Haiman, N.~Loehr, J.~B. Remmel, and A.~Ulyanov.
\newblock A combinatorial formula for the character of the diagonal
  coinvariants.
\newblock \emph{Duke J. Math.}, 126:\penalty0 195--232, 2005.

\bibitem[Konheim and Weiss(1966)]{KW}
A.~G. Konheim and B.~Weiss.
\newblock An occupancy discipline and applications.
\newblock \emph{SIAM J. Applied Math.}, 14:\penalty0 1266--1274, 1966.

\bibitem[Macdonald(1995)]{Macdonald}
I.~G. Macdonald.
\newblock \emph{Symmetric functions and Hall polynomials}.
\newblock Oxford Mathematical Monographs, New York, 2 edition, 1995.

\end{thebibliography}
\label{sec:biblio}

\end{document}